\documentclass[12pt, twoside]{article}
\usepackage{amsmath,amsthm,amssymb}
\usepackage{amsrefs}
\usepackage{times}
\usepackage{enumerate}
\usepackage{tikz}
\usepackage{hyperref}
\usepackage{mathtools}

\DeclarePairedDelimiter\abs{\lvert}{\rvert}
\DeclarePairedDelimiter\norm{\lVert}{\rVert}
\newcommand{\eps}{\varepsilon}

\newcommand{\ignore}[1]{}

\newcommand{\R}{\mathbb{R}}

\newcommand{\ch}{\boldsymbol{1}}
\newcommand{\dx}{d\textbf{x}}
\newcommand{\sgn}{\mathrm{sgn}\,}

\theoremstyle{definition}
\newtheorem{thm}{Theorem}[section]

\newtheorem{lemma}[thm]{Lemma}

\newtheorem{rem}[thm]{Remark}

\newtheorem{notation}[thm]{Notation}
\newtheorem{proposition}[thm]{Proposition}

\newtheorem{mainthm}[thm]{Main Theorem}

\numberwithin{equation}{section}

\newcommand{\subjclass}[1]{\bigskip\noindent\emph{2010 Mathematics Subject Classification: }\enspace#1}
\newcommand{\keywords}[1]{\noindent\emph{Keywords:}\enspace#1}

\begin{document}


\baselineskip=17pt


\title{Relaxation to a planar interface in the\\ Mullins-Sekerka problem}

\author{Olga Chugreeva\\
RWTH Aachen University\\
olga@math1.rwth-aachen.de\\
Felix Otto\\
MPI for Mathematics in the Sciences, Leipzig, Germany\\
felix.otto@mis.mpg.de\\
Maria G. Westdickenberg\\
RWTH Aachen University\\
maria@math1.rwth-aachen.de}

\date{5 October, 2017}

\maketitle


\begin{abstract}
We analyze the convergence rates to a planar interface in the Mullins-Sekerka model by applying a relaxation method based on relationships among distance, energy, and dissipation. The relaxation method was developed by two of the authors in the context of the 1-d Cahn-Hilliard equation and the current work represents an extension to a higher dimensional problem in which the curvature of the interface plays an important role. The convergence rates obtained are optimal given the assumptions on the initial data.

\subjclass{Primary 35B35; Secondary 35Q99,35R35.}

\keywords{Mullins-Sekerka; energy method; relaxation rates; planar profile.}
\end{abstract}

\section{Introduction}
\label{s:Intro}
The Mullins-Sekerka model \cite{MS1} for the evolution of phase interfaces has been fundamental in developing an understanding of solidification in pure liquids and binary alloys. The surrounding literature is vast; we refer to the classical works \cite{MS2,L} and the thousands of citing references. Mathematically, quantifying the behavior of solutions of the Mullins-Sekerka problem, a prototypical nonlocal free boundary problem, is interesting in its own right and because Mullins-Sekerka arises in a sharp-interface limit of
the Cahn-Hilliard equation, as demonstrated in the seminal paper \cite{P} and rigorously established in \cite{ABC}.

The quasistatic Mullins-Sekerka model with Gibbs-Thompson boundary condition in $d=2$ consists of the dynamics of an interface $\Gamma$ evolving with the normal velocity
\begin{align}
  V=- \big[\nabla f\cdot n\big]
 ,\label{vel}
\end{align}
where
$f:\R^2\to\R$ satisfies
\begin{align}\label{ms}
\begin{cases}
\Delta f=0&\text{in } \ \R^2\setminus{\Gamma},\\
  f=\kappa&\text{on }\ \Gamma.
  \end{cases}
\end{align}
Here $\kappa$ is the curvature of $\Gamma$, $n$ is the unit outward normal vector to the interface, and $[\nabla f\cdot n]$ is the jump in the normal derivative across the interface, cf.~\eqref{jump} below.

We are interested in capturing relaxation rates for an interface $\Gamma$ that converges in time to a planar interface. We work under the assumption that the interface can be parameterized as the graph of a smooth function $h$ via
\begin{align}
\Gamma(t)=\{(x,z)\in \R^{2}\colon z=h(x,t),\, x\in \R\}.\label{param}
\end{align}
Without loss of generality, we choose the orientation of $\Gamma$ so that $n$ is the unit outward normal to the region $\Omega_{+}:=\{(x,z)\in\R^{2}\colon z>h(x)\}$ lying to the right of~$\Gamma$, and $\sgn V=-\sgn \kappa=\sgn h_{xx}$; see Figure~\ref{figurezz} and Section \ref{S:prelim}.

\begin{figure}
\centering
\begin{tikzpicture}[yscale=1.1, xscale=1.5]

\draw [black, very thick, ->] (4, 5) -- (4,-1);
\draw[blue, thick, domain=-0.5:4.5, samples=500] plot [smooth] ({4+0.02*(\x-4.5)*(\x-4.5)*(\x-2.5)*(\x-0.5)*(\x+0.5)*(\x+0.5)},\x);
\draw [red, very thick, ->] (3.25, 1.63) -- (3.75, 1.63);
\node [red, left] at (3.1, 1.65) {$V$};
\draw [red, very thick, <-] (4.25, 3.45) -- (4.95, 3.45);
\node [red, right] at (5, 3.45) {$V$};
\node [left] at (3.9,-1) {$x$};
\node [blue] at (4.65, 4.25) {\large{$\Gamma$}};
\node at (2, 3.7) {\large{$\Omega _{-}(t)$}};
\node at (6, 3.7) {\large{$\Omega _{+}(t)$}};
\node at (6, 1) {$\Delta f=0$};
\node at (2, 1) {$\Delta f=0$};
\node [right, blue] at (4.1, 0.1) {$f=\kappa \text{ on }\Gamma$};
\end{tikzpicture}
\caption{A sketch of an admissible interface $\Gamma$. The normal is chosen outward with respect to $\Omega_+$. See Section \ref{S:prelim} for details.}\label{figurezz}
\end{figure}
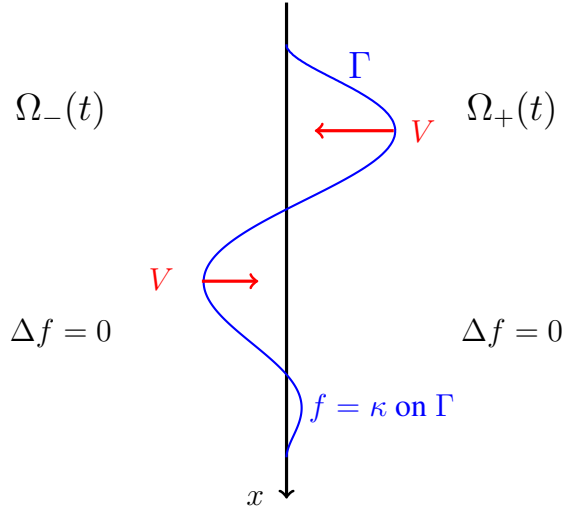

The Mullins-Sekerka problem has the structure of a gradient flow with the triad of squared distance $H$, energy gap $E$, and dissipation $D$ defined via
\begin{align}
H:&=\int_{\R^2} |\nabla \varphi|^2 \,d\textbf{x}\qquad \text{for }-\Delta\varphi=\textbf{1}_{\Omega_+(t)}-\textbf{1}_{\Omega_0},\label{h}\\
E:&=\int_\R \Big( \sqrt{1+h_x^2}-1\Big)\,dx,\label{e}\\
D:&=\int_{\R^2} |\nabla f|^2\,\dx,\label{d}
\end{align}
where $\Omega_+(t)$ indicates the area to the right of $\Gamma$ at time $t$ and $\Omega_0$ represents the area to the right of the $x$-axis. In \cite{OW}, a method for establishing convergence rates to equilibrium based on algebraic and differential relationships among distance, energy, and dissipation is developed in the context of the 1-d Cahn-Hilliard equation. This relaxation framework builds on observations of Brezis \cite{B} for gradient flows with respect to a convex energy and establishes an extension to the mildly nonconvex setting. It is one of the main ingredients in the study of metastability of the 1-d Cahn-Hilliard equation carried out in~\cite{SW} and is used in \cite{E} to study convergence to equilibrium for the thin-film equation. Here we apply the relaxation framework to the Mullins-Sekerka model.

\begin{rem}[Scale invariance and the smallness condition]
The Mullins-Sekerka model has the following scaling properties. If one rescales $\textbf{x}=\lambda \hat{\textbf{x}}$ and $t=\lambda^3\hat{t}$, then
\begin{align*}
h=\lambda \hat{h},\qquad  H=\lambda^4\hat{H},\qquad E=\lambda\hat{E},\qquad D=\lambda^{-2}\hat{D}.
\end{align*}
Hence $E^2D$ is scale invariant, and it is reasonable to impose a smallness condition on this quantity, as we will do. Also $|h_x|$ is scale invariant, and we deduce smallness of the slope from smallness of $E^2D$; cf. Remark \ref{rem:opt} below.

Dimension $d=4$ seems to be the critical dimension at which the method of this paper breaks down (formally, $E\sim \lambda^3$ and $D\sim 1$). It may be possible to use the method of this paper for $d=3$ with $ED^2$ playing the role of the scale invariant quantity, however this would certainly be more complicated. Here we consider $d=2$ as the simplest setting in which geometric effects such as curvature play a role.
\end{rem}

We assume that there exists a global in time, sufficiently regular solution of \eqref{vel}-\eqref{ms}. On the level of the parametrization~\eqref{param}, this implies in particular that $h(\cdot,t)\in H^{2}(\R)$ for all $t$ and that $h$ is a smooth function of~$t$.
For local and in some cases global existence on bounded domains, we refer to \cite{C} for weak solutions and to \cite{CHY,ES,ES2} for smooth solutions.
Here, we consider $d=2$ and an interface that is a perturbation of the $x$-axis. We show that smooth solutions satisfy optimal convergence rates in time; see Subsection \ref{ss:opt} for a discussion of optimality.

\begin{notation}\label{notn}
 We will use the following notation throughout the work.
We write
\begin{align*}
A\lesssim B
\end{align*}
for two quantities $A$ and $B$ if there exists a universal constant $C\in(0, \infty)$ such that $A\leqslant CB$.

We write
\begin{align*}
A\sim B
\end{align*}
 if $A\lesssim B$ and $B\lesssim A$.

We say that
\begin{align*}
A\ll B \ \text{ implies } \ \tilde A\ll \tilde B \ (\text{or }\tilde A\lesssim  \tilde B,\text{ respectively})
\end{align*}
 if for every $\eps>0$ there exists a $\delta>0$ such that $A\leqslant \delta B$ implies $\tilde A\leqslant \eps \tilde B$ (or if there exist $C\in(0, \infty)$ and $\delta>0$ such that $A\leqslant \delta B$ implies $\tilde A\leqslant C\tilde B$, \text{ respectively}).
\end{notation}

Our main result is as follows.
\begin{mainthm}\label{t:MS}
Consider a global solution $(f,\Gamma)$ of \eqref{vel}-\eqref{ms}. Let $H(t)$, $E(t),\,D(t),\,h(\cdot,t)$ represent the squared distance, energy, dissipation, and $z$-coordinate of $\Gamma$ as functions of time, and let $H_0:=H(0)$, $E_0:=E(0),$ $D_0:=D(0)$. Under the assumptions that
\begin{align}
E_0^2D_0\ll 1\label{assm}
\end{align}
and
\begin{equation}
\sup_{\R}|h_{x}(\cdot, 0)|\leqslant 1, \label{assmh}
\end{equation}
there holds
\begin{align}
  (E^2D)(t)\ll 1\qquad \text{for all }t\geqslant 0.\label{regime}
\end{align}
As a consequence, the slope remains small for all time, i.e.,
\begin{align}
   \sup_{\R}|h_x(\cdot,t)|&\ll 1\qquad\text{for all }t\geq 0,\label{hxsmall}
\end{align}
and the energy and squared distance obey
 \begin{align}
  E(t)&\lesssim \min\left\{E_0,\frac{H_0}{t}\right\},\label{eless}\\
  H(t)&\lesssim H_0.\label{hless}
 \end{align}
 Moreover, for times of the order $t\gg H_0^{3/4}$, there holds
 \begin{align}\label{dless}
  D(t)\lesssim \frac{H_0}{t^2}.
 \end{align}
\end{mainthm}

\begin{rem}\label{rem:opt}
From Theorem~\ref{t:MS} one can read off relaxation rates for the slope~$h_{x}$ and the height $h$. The key estimate for the slope is provided by Lemma~\ref{l:stays} in the form
\begin{align}
  \sup_\R\,|h_x(\cdot,t)|\lesssim (E^2(t)D(t))^{1/6}\quad\text{if }|h_x(\cdot,t)|\leq 1.\label{h0912}
\end{align}
On the one hand, \eqref{h0912} and \eqref{regime} deliver \eqref{hxsmall}. On the other hand, \eqref{h0912} together with \eqref{hxsmall}, \eqref{eless}, and \eqref{dless} yields the relaxation rate
\begin{align}\label{hxt}
\sup_\R\,|h_x(\cdot,t)|\lesssim \frac{H_0^{1/2}}{t^{2/3}}\qquad\text{ for } t\gg H_0^{3/4}.
\end{align}

To obtain relaxation of the height $h$, we make use of the interpolation inequality
\begin{align*}
  \sup|h|\lesssim \left(\int_\R \big( |\partial_x|^{-1/2} h\big)^2\,dx\right)^{1/6}\left(\int_\R h_x^2\,dx\right)^{1/3};
\end{align*}
see \eqref{eq:DefFrac} for the definition of the fractional Sobolev norms.
An application of
Lemmas \ref{l:2p1} and \ref{l:2p4} below then yields
\begin{align*}
  \sup_\R |h(\cdot,t)|\lesssim (H(t)E^{2}(t))^{1/6}\quad\text{if }|h_x(\cdot,t)|\leq 1,
\end{align*}
which in light of \eqref{hxsmall}, \eqref{eless}, and \eqref{hless} implies
\begin{align}\label{ht}
\sup_\R\,|h(\cdot,t)|\lesssim \min \Big\{(H_{0}E_{0}^{2})^{1/6}, \frac{H_0^{1/2}}{t^{1/3}}\Big\}.
\end{align}
\end{rem}

\begin{rem}
 Since we assume that the initial data satisfies $|h_x|\leq 1$, all of our constants are universal. One could instead assume $|h_x|\leq C<\infty$, but then the constants would depend on $C$.
\end{rem}

We prove Theorem \ref{t:MS} in Section \ref{S:ode} after first establishing the necessary algebraic and differential relationships among $H,\,E,\,\text{and }D.$
\subsection{Optimality}\label{ss:opt}
We now give a formal argument for the optimality of the above rates.
Our first observation is that the decay estimates (\ref{eless}) and (\ref{dless}) are optimal.
It is well-known, and easily seen from (\ref{kaph}) and (\ref{vandh}), that the linearization
of Mullins-Sekerka around the flat interface is given by the third-order parabolic equation
\begin{align}\label{h1d}
h_t+|\partial_x|^3h=0.
\end{align}
Fourier transforming in space yields
\begin{align*}
\||\partial_x|^nh\|_2^2=\int_{\mathbb{R}}|k|^{2n}\exp(-2t|k|^3)|\hat h_0|^2dk,
\end{align*}
which allows us to write
\begin{align*}
tE&=t\|\partial_xh\|_2^2=\int_{\mathbb{R}}|t^{1/3}k|^{3}\exp(-2|t^{1/3}k|^3)\frac{1}{|k|}|\hat h_0|^2dk,\\
t^2D&=t^2\||\partial_x|^{5/2}h\|_2^2=\int_{\mathbb{R}}|t^{1/3}k|^{6}\exp(-2|t^{1/3}k|^3)\frac{1}{|k|}|\hat h_0|^2dk.
\end{align*}
In view of
\begin{align}\label{230817}
&H_0=\int_{\mathbb{R}}\frac{1}{|k|}|\hat h_0|^2dk\\
\mbox{and}\quad
&\sup_{t\ge 0}|t^{1/3}k|^{n}\exp(-2|t^{1/3}k|^3)=c(n)\in(0,\infty),
\end{align}
we see that we cannot expect better than $tE,t^2D\lesssim H_0$, cf.\  (\ref{eless}) and (\ref{dless}).

\medskip

Our second observation is that the rates (\ref{hxt}) and (\ref{ht}) are optimal for compactly supported data with $\int_{\mathbb{R}}h_0\,dx\not=0$
and ``almost optimal''
for compact perturbations under our assumption of $H_0<\infty$. The pointwise time-asymptotic behavior for (\ref{h1d}) for compactly supported
initial data $h_0$ is governed
by its ``heat kernel'', which in view of the conservative form of (\ref{h1d}) and its scale invariance
must be of the form $t^{-1/3}G(\frac{x}{t^{1/3}})$, with a smooth (but only algebraically decaying)
mask $G$ given by $\hat G(k)=\exp(-|k|^3)$. This shows that for compactly supported initial data $h_0$
the decay rates $\sup_{\mathbb{R}}|h_x(\cdot,t)|=O(t^{-2/3})$ and $\sup_{\mathbb{R}}|h(\cdot,t)|=O(t^{-1/3})$,
cf.\  (\ref{hxt}) and (\ref{ht}), are optimal. These rates hold provided $\int_{\mathbb{R}}h_0\,dx\not=0$, whereas for
$\int_{\mathbb{R}}h_0\,dx=0$ they improve by an order of $O(t^{-1/3})$. Note that in view of (\ref{230817}), our
assumption $H_0<\infty$ implies $\int_{\mathbb{R}}h_0\,dx=\hat h_0(0)=0$, but this enforcement of vanishing average
is only borderline since the integral $\int_{\mathbb{R}}\frac{1}{|k|}dk$ diverges only logarithmically at $k=0$.
Hence while the rates (\ref{hxt}) and (\ref{ht}) are not optimal under our assumption within the class of
compactly supported initial data, they would be if we minimally changed our assumption to
$\int_{\mathbb{R}}(|\partial_x|^{-\alpha}h_0)^2dx<\infty$ with $\alpha<\frac{1}{2}$.

\subsection{Previous results}
We are not aware of previous results on algebraic-in-time rates of relaxation to a planar front for the Mullins-Sekerka problem. Our results are related to those for relaxation to a planar front in the Cahn-Hilliard equation in \cite{H}---and similarly for $d\geq 3$ in \cite{KKT}---although the smallness condition that is assumed in those works is such that the perturbation vanishes in the sharp-interface limit. A similar remark holds for the relaxation to planar fronts in the related nonlocal equation analyzed in \cite{CO}.

Recently an application of the relaxation framework from \cite{OW} to the multi-dimensional Cahn-Hilliard equation for initial data that are close in $L^\infty$ to the planar profile was carried out in \cite{DGS}. The assumption of $L^\infty$ closeness (which they show is preserved by the evolution) allows them to use the linear energy-energy-dissipation estimates to deduce nonlinear estimates. A more ambitious goal would be to analyze the evolution for initial data far from the planar profile. On the one hand, existence methods for the Cahn-Hilliard equation based on existence of ``good'' solutions to the Mullins-Sekerka evolution (see \cite{ABC} and also \cite{CCO}) may help to derive from the present paper estimates for the relaxation of solutions to the higher dimensional Cahn-Hilliard equation for well-prepared initial data close to their (nonplanar) sharp-interface limit. On the other hand, an even more challenging problem would be to handle initial data that are not well-prepared and for which one expects an initial layer in time: On a short timescale, the initial data relaxes to a small neighborhood of its sharp-interface limit and on longer timescales, the evolution approximates the Mullins-Sekerka evolution. This would require analyzing the energy-energy-dissipation estimates in the presence of \emph{both} perturbations away from the sharp-interface limit \emph{and} perturbations due to the geometric effects of non-planarity---and would be interesting to consider in future work.

Our results are also related to previous results on stability of and exponential convergence toward equilibrium solutions. Global existence and exponential convergence to a circular interface for initial data that is asymptotically close to a circle is established in \cite{C}; the methods include use of a regularized problem and a priori estimates. Later in \cite{ES2}, global existence of classical solutions and exponential convergence to spheres in $d\geq 2$ for initial data close to a sphere is shown; the method relies on center manifold theory. More recently, there has been interest in global existence for a related problem known as the Muskat problem; we mention for instance \cite{CC,PS} and the references cited there.

\subsection{Organization} In Section \ref{S:prelim} we fix notation. The algebraic relationships that we will need for our result are collected and proved in Section \ref{S:alg} and the differential relationships in Section \ref{S:diff}. Finally, in Section \ref{S:ode}, we use these relationships together with an ODE argument to prove Theorem \ref{t:MS}. Through most of Sections \ref{S:alg} and \ref{S:diff} we  assume that $|h_x|\leq 1$ and later in the proof of Theorem \ref{t:MS} we  show that this condition holds true for the solution. The appendix contains some technical facts about fractional Sobolev spaces that we use in the proofs.

\section{Preliminaries}\label{S:prelim}
We remark for reference below that the normal vector $n$, the curvature $\kappa$, and the normal velocity $V$ can be expressed in terms of the parameterization~\eqref{param} as
\begin{align}
  n&=\frac{(h_x,-1)}{\sqrt{1+h_x^2}},\notag\\
  \kappa&=\frac{d}{dx}\frac{h_x}{\sqrt{1+h_x^2}}=\frac{h_{xx}}{(\sqrt{1+h_x^2})^3},\label{kaph}\\
  V&=-\frac{h_{t}}{\sqrt{1+h_x^2}}.\label{vandh}
\end{align}
Furthermore, under the assumption that $|h_{x}|\leqslant 1$, the line integral element on the curve $\Gamma$ satisfies
\begin{align}
ds=\sqrt{1+h_x^2}\,dx \sim dx.  \label{lineint}
\end{align}
We will occasionally denote by $\R^{2}_{+}$ the right half-plane:
\begin{align*}
\R^{2}_{+}=\{(x, z)\in \R^{2}: z>0\}
\end{align*}
and by
\begin{align*}
  \dx=dx\,dz
\end{align*}
the area element.

For a function $g$ that is continuous on the closures of both regions $\Omega_{+}$ and $\Omega_{-}$, we denote by
\begin{align}
  \text{$g_{+}$ and $g_{-}$}\label{notnpm}
\end{align}
the restriction to $\Gamma$ coming from the respective region. We denote the jump in $g$ across $\Gamma$ by
\begin{align}\label{jump}
[g]:=g_{+}-g_{-}.
\end{align}

\section{Algebraic lemmas}\label{S:alg}
In this section, we collect the fundamental algebraic relationships that we will use.
For our definition of the homogeneous fractional Sobolev spaces and a summary of the facts that we will need here and below, we refer to Appendix~\ref{S:fracsp}.

We begin by relating the curvature $\kappa$ and the height $h$ to energy and dissipation.
\begin{lemma}\label{l:2p1}
  Under the assumption that
  \begin{align}
    \sup_\R|h_x|\leq 1,\label{hx11uh}
  \end{align}
  there holds
\begin{align}
  \int_\Gamma\left( |\partial_s|^{1/2}\kappa\right)^2\,ds  &\lesssim D,\label{l1.1}\\
  \int_\Gamma\left( |\partial_s|^{-1}\kappa\right)^2\,ds&\lesssim\int_\R h_x^ 2\,dx\sim E,\label{l1.2}\\
\int_\Gamma \kappa^2\,ds&\sim  \int_\R h_{xx}^2\,dx\lesssim E^{1/3}D^{2/3}.\label{l1.3}
  \end{align}
\end{lemma}
\begin{proof}

We first consider \eqref{l1.1}. Because of assumption \eqref{hx11uh}, we can straighten the curve $\Gamma$ by introducing the new variable $\tilde{z}=z-h(x)$. This transforms \eqref{ms} to the problem
\begin{align*}
\begin{cases}
  -\tilde{\nabla}\cdot a\tilde{\nabla} f=0, & (x,\tilde{z})\in \R\times(0,\infty),\\
 f=\kappa, &(x,\tilde{z})\in\R\times\{\tilde{z}=0\},
\end{cases}
\end{align*}
where
\begin{align*}
  a:=\left(\begin{matrix}
    1&-h_x(x)\\
    0&1
  \end{matrix}\right)^{T}\left(\begin{matrix}
    1&-h_x(x)\\
    0&1
  \end{matrix}\right)
\end{align*}
is uniformly positive definite because of \eqref{hx11uh}. Therefore the transformed problem is uniformly elliptic and by the trace estimate, the solution satisfies
\begin{align}
  \int_\R\left( |\partial_x|^{1/2}\kappa\right)^2\,dx=\int_\R\left( |\partial_x|^{1/2}f(x,0)\right)^2\,dx\lesssim \int_{\R^{2}_{+}}|\tilde\nabla f|^2\,dx\,d\tilde{z}.\label{ad.1}
\end{align}
On the other hand, because of assumption \eqref{hx11uh} there holds
\begin{align}
  \int_\Gamma\left( |\partial_s|^{1/2}\kappa\right)^2\,ds  &\sim\int_\R\left( |\partial_x|^{1/2}\kappa\right)^2\,dx\label{ad.2}\\
\text{and}\quad  \int_{\R^{2}_{+}}\abs{\tilde{\nabla} f}^2\,dx\,d\tilde{z}&\lesssim\int_{\R^2}\left|\nabla f\right|^2\,dx\,dz;\label{ad.3}
\end{align}
cf.\ \eqref{eq:frac_sim}.
The combination of \eqref{d}, \eqref{ad.1}, \eqref{ad.2}, and \eqref{ad.3} leads to \eqref{l1.1}.

We now address \eqref{l1.2}. Under assumption \eqref{hx11uh}, we have
\begin{align}
  |\theta|\sim |h_x|,\label{thet}
\end{align}
where $\theta$ is the angle between the tangent line to $\Gamma$ and the $x$-axis. Recalling that $\kappa=\frac{d\theta}{ds}$, we find for an arbitrary $\zeta\in \dot{H}^{1}(\Gamma)$ that
\begin{eqnarray*}
  \int_\Gamma\kappa\,\zeta\,dx&=&-\int_\Gamma \theta\,\zeta_s \,ds\\
 &\leq &\left(\int_\Gamma \theta^2\,ds\,\int_\Gamma \zeta_s ^2\,ds\right)^{1/2}\\
  &\overset{\eqref{thet}}\lesssim &\left(\int_\R h_x^2\,dx\,\int_\Gamma \zeta_s ^2\,ds\right)^{1/2}.
\end{eqnarray*}
From this estimate we deduce the first inequality in \eqref{l1.2} (via duality in $\dot{H}^{-1}$). The second part of \eqref{l1.2} is a direct consequence of definition \eqref{e},  assumption \eqref{hx11uh}, and the identity
\begin{align*}
h_{x}^{2}=(\sqrt{1+h_{x}^{2}}-1)(\sqrt{1+h_{x}^{2}}+1).
\end{align*}

The first part of \eqref{l1.3} follows from formula \eqref{kaph} and assumption \eqref{hx11uh}. The second part of \eqref{l1.3}, on the other hand, follows from interpolation in terms of~$\kappa$ using \eqref{l1.1} and the second part of \eqref{l1.2} (cf.\ Appendix~\ref{S:fracsp}).
\end{proof}
We will now use the previous lemma to show that $\sup|h_x|$ is controlled by the quantity $E^2D$. Later in Lemma \ref{l:invariant}, we will show that if $E^2D$ is small initially, then it stays small for all time.
\begin{lemma}\label{l:stays}
  Under the assumption that
  \begin{align}
    \sup_\R|h_x|\leq 1,\label{hx1}
  \end{align}
  there holds
  \begin{align}
    \sup_\R |h_x|\lesssim (E^2D)^{1/6}.\label{ed2}
  \end{align}
\end{lemma}
\begin{proof}
We recall the interpolation inequality
  \begin{equation}
    \sup_\R |h_x|\lesssim \left( \int_\R h_x^2\,dx \int_\R h_{xx}^2\,dx\right)^{1/4}\overset{\eqref{l1.2}}{\lesssim} \left( E\,\int_\R h_{xx}^2\,dx\right)^{1/4}.\notag
  \end{equation}
Plugging the second part of \eqref{l1.3} into the right-hand side gives \eqref{ed2}.
\end{proof}

Our next goal is to measure the $\dot{H}^{-1/2}$ seminorm of $h$ in terms of the squared distance $H$. First we need a preliminary lemma.
\begin{lemma}\label{littlewood}
For every function $h\in \dot{H}^{-1/2}\cap H^{2}(\R)$, there holds
\begin{align}
\norm{h}_{3}\lesssim \norm{|\partial_x|^{-1/2}h}_{2}^{2/3}\norm{h_x}_{2}^{1/6}\norm{h_{xx}}_{2}^{1/6}.\label{bess}
\end{align}
\end{lemma}
\begin{proof}
We note that \eqref{bess} is a consequence of the elementary estimates
\begin{align}
  \norm{h}_3&\lesssim \norm{h}_\infty^{1/3}\norm{h}_2^{2/3},\label{f1}\\
  \norm{h}_\infty&\lesssim\norm{h}_2^{1/2}\norm{h_x}_2^{1/2}\label{f2}
 \end{align}
 and of the interpolation estimate
  \begin{equation}
  \norm{h}_2\lesssim \norm{|\partial_x|^{-1/2}h}_2^{4/5}\norm{h_{xx}}_2^{1/5},\label{f3}
  \end{equation}
cf.\ the appendix.
Indeed, inserting \eqref{f2} into \eqref{f1}, we obtain
\begin{align}
  \norm{h}_3\lesssim\norm{h_x}_2^{1/6}\norm{h}_2^{5/6}.\label{f4}
\end{align}
Inserting \eqref{f3} into \eqref{f4} gives \eqref{bess}.
\end{proof}

Lemma \ref{littlewood} together with smallness of $E^2D$ yields a bound on a negative norm of $h$.
\begin{lemma}\label{l:2p4}
  Under the assumption that
  \begin{align}
    \sup_\R|h_x|\leq 1\qquad\text{and}\qquad E^2D\ll 1,\label{assed}
  \end{align}
  there holds
  \begin{align}
    \int_\R\left( |\partial_x|^{-1/2}h\right)^2\,dx\lesssim H.\label{hH}
  \end{align}
\end{lemma}

\begin{proof}
  On the one hand, using duality for the left-hand side of \eqref{hH} and recalling Proposition \ref{prop:harmext}, there holds
  \begin{align}
   \left(\int_\R\left( |\partial_x|^{-1/2}h\right)^2\,dx\right)^{1/2}
   =\sup\left\{\frac{\int_\R h(x)\zeta(x,0)\,dx}{(\int_{\R^2_{+}}|\nabla \zeta|^2\,\dx)^{1/2}}\bigg|\zeta \text{ harmonic on }\R^2_{+}\right\}.\label{hminus}
  \end{align}

  On the other hand,
  for any $\zeta\in \dot{H}^1(\R^2)$ there holds
  \begin{align}
&  \left|\int_{\{h>0\}}\int_0^{h(x)}\zeta(x,z)\,dz\,dx-\int_{\{h<0\}}\int_{h(x)}^0\zeta(x,z)\,dz\,dx \right|\notag\\
& \qquad  \leq H^{1/2}\left(\int_{\R^2}|\nabla \zeta|^2\,dx\,dz\right)^{1/2}.\label{Hzeta}
  \end{align}

We would like to replace the right-hand side of \eqref{hminus} by the left-hand side of \eqref{Hzeta} for a function $\zeta \in \dot{H}^1(\R^2)$ that is harmonic on $\R^{2}_{+}$. Hence we need to measure the difference. Considering without loss of generality the integral over $\{h>0\}$, we estimate:
\begin{align}
 &\left|\int_{\{h>0\}} \int_0^{h(x)} \zeta(x,z)-\zeta(x,0)\,dz\,dx\right|\notag\\
 &\quad=\left|\int_{\{h>0\}} \int_0^{h(x)} \int_{0}^{z}\partial_{z}\zeta(x,z')\,dz'\,dz\,dx\right|\notag\\
  &\quad=\left|\int_{\{h>0\}} \iint \ch_{\{0<z<h(x)\}} \ch_{\{0<z'<z\}} \partial_z\zeta(x,z')\,dz\,dz'\,dx\right|\notag\\
  &\quad=\left|\int_{\{h>0\}}\int \ch_{\{0<z'<h(x)\}}(h(x)-z') \partial_z\zeta(x,z')\,dz'\,dx\right|\notag\\
  &\quad\leq \left(\int_{\{h>0\}}\int_0^{h(x)} (h(x)-z')^2 \,dz'\,dx \int_{\R^2}|\nabla\zeta|^2\,dz'\,dx \right)^{1/2}\notag\\
  &\quad\lesssim\left(\int_{\{h>0\}} |h|^3\,dx  \int_{\R^2}|\nabla\zeta|^2\,dz'\,dx \right)^{1/2}.\label{b1}
\end{align}

The combination of \eqref{hminus}, \eqref{Hzeta}, and \eqref{b1} leads to
\begin{align*}
   \left(\int_\R\left( |\partial_x|^{-1/2}h\right)^2\,dx\right)^{1/2}\lesssim H^{1/2}+\left(\int_{\{h>0\}} |h|^3\,dx  \right)^{1/2}.
\end{align*}

It remains to show that
\begin{align}
 \int_\R |h|^3\,dx\ll \int_\R\left( |\partial_x|^{-1/2}h\right)^2\,dx.\label{much}
\end{align}
To establish \eqref{much}, we combine \eqref{bess}, \eqref{l1.2}, and \eqref{l1.3} to deduce
\begin{eqnarray*}
    \int_\R |h|^3\,dx&\lesssim &(E^2D)^{1/6} \int_\R\left( |\partial_x|^{-1/2}h\right)^2\,dx\\
    &\overset{\eqref{assed}}\ll & \int_\R\left( |\partial_x|^{-1/2}h\right)^2\,dx.
\end{eqnarray*}

\end{proof}

\begin{lemma} Under the assumption that
  \begin{align}
    \sup_\R|h_x|\leq 1\qquad\text{and}\qquad E^2D\ll 1,\notag
  \end{align}
  we have the interpolation inequality
  \begin{align}
    E\lesssim (HD)^{1/2}.\label{ehd}
  \end{align}
\end{lemma}
\begin{proof}
We estimate
\begin{eqnarray*}
  E&\overset{\eqref{l1.2}}\sim&\int_\R h_x^2\,dx\leqslant\left(\int_\R h_{xx}^2\,dx\right)^{3/5}\left(\int_\R (|\partial_x|^{-1/2}h)^2\,dx\right)^{2/5}\\
  &\overset{\eqref{l1.3},\eqref{hH}}\lesssim &\left(E^{1/3}D^{2/3}\right)^{3/5}H^{2/5}=E^{1/5} \,(HD)^{2/5}.
\end{eqnarray*}
We obtain \eqref{ehd} by dividing by $E^{1/5}$ and taking the power $5/4$.
\end{proof}


\section{Differential lemmas}\label{S:diff}
In this section we establish the differential relationships that we will need.
\begin{lemma}\label{l:diff}  Under the assumption that
  \begin{align}
    \sup_\R|h_x|\leq 1,\qquad E^2D\ll 1,\label{hxone}
  \end{align}
  the following differential relationships hold:
  \begin{align}
&    \frac{dE}{dt}=-D,\label{ee}\\
 &   \frac{dD}{dt}+\int_\Gamma V_s^2\,ds\lesssim D^{5/2}+ED^3,\label{dd}\\
   & \frac{dH}{dt}\lesssim H^{1/2}\,E^{1/6}\,D^{7/12}.\label{dhdt}
  \end{align}
\end{lemma}
\begin{proof}
For \eqref{ee} we calculate directly that
\begin{eqnarray*}
 \frac{d}{dt}\int_\R \Big(\sqrt{1+h_x^2}-1\Big)\,dx&=&\int_\R\frac{h_x\,h_{xt}}{\sqrt{1+h_x^2}}\,dx\\
 &\overset{\eqref{kaph}-\eqref{vandh}}=&\int_\Gamma \kappa\,V\,ds\\
 &\overset{\eqref{ms},\eqref{vel}}=&-\int_{\Gamma} f\,\big[\nabla f\cdot n\big]\,ds\\
 &=&-\int_{\R^2}|\nabla f|^2\,\dx.
\end{eqnarray*}

In order to show \eqref{dd}, we recall the equation for the time evolution of the curvature of an interface. Letting $V$ denote the normal velocity and $s$ denote the arc-length parameter, one can check that the full (convective) time derivative of the curvature is
\begin{align}
  \frac{D\kappa}{dt}=-V_{ss}-\kappa^2\,V;\label{noty}
\end{align}
see for instance \cite{KTZ}[p.443, setting $\alpha=0$]. Directly computing the time-derivative of the dissipation gives
\begin{eqnarray*}
\frac{d}{dt}D
&=&\frac{d}{dt}\int_{\R^2}|\nabla f|^2\,\dx=\frac{d}{dt}\int_{\Omega_{+}(t)}|\nabla f_{+}|^2\,\dx+\frac{d}{dt}\int_{\Omega_{-}(t)}|\nabla f_{-}|^2\,\dx\\
&=&\int_{\Omega_{+}(t)}\frac{d}{dt}|\nabla f_{+}|^2\,\dx+\int_{\Omega_{-}(t)}\frac{d}{dt}|\nabla f_{-}|^2\,\dx +\int_{\Gamma}V (|\nabla f_{+}|^2-|\nabla f_{-}|^2)\,ds\\
&\overset{\eqref{ms}}=&\int_\Gamma 2\frac{Df}{dt}\Big[\nabla f\cdot n\Big]+V\Big(|\nabla f|_+^2-|\nabla f|_-^2\Big)\,ds\\
&\overset{\eqref{vel}, \eqref{ms}}=&\int_\Gamma -2\frac{D\kappa}{dt}V+V\Big(|\nabla f|_+^2-|\nabla f|_-^2\Big)\,ds,
\end{eqnarray*}
where we have used the notation \eqref{notnpm} and \eqref{jump}.
We remark that $[|\nabla f|^2]=[(\nabla f\cdot n)^2]$ on $\Gamma$. This fact together with \eqref{noty},\eqref{vel}, and an integration by parts leads to
\begin{align}
 \frac{d}{dt}D +\int_\Gamma 2V_s^2\,ds&
 =\int_\Gamma 2\kappa^2 V^2-\big( (\nabla f\cdot n)_+ + (\nabla f\cdot n)_-\big)V^2\,ds.\label{deq}
\end{align}
Hence to show \eqref{dd}, it suffices to control the right-hand side.
We begin with the preliminary estimates
\begin{align}
\int_\Gamma\Big(|\partial_s|^{-1/2}\,V\Big)^2\,ds&\lesssim D,\label{dec1}\\
\norm{V}_{\infty}^2&\lesssim\,D^{1/3}\,\left(\int_\Gamma V_s^2\,ds\right)^{2/3}\label{dec2}.
\end{align}
We obtain \eqref{dec2} from \eqref{dec1} and the  interpolation estimate
\begin{align*}
\norm{V}_{\infty}^2&\lesssim \left(\int_\Gamma V^2\,ds\right)^{1/2}\left(\int_\Gamma V_s^2\,ds\right)^{1/2}\notag\\
&\lesssim\,\left(\int_\Gamma\Big(|\partial_s|^{-1/2}\,V\Big)^2\,ds\right)^{1/3}\left(\int_\Gamma V_s^2\,ds\right)^{2/3};
\end{align*}
cf.\ Appendix~\ref{S:fracsp}.
On the other hand we will deduce \eqref{dec1}  from \eqref{vel} and
\begin{align}\label{decc}
\int_\Gamma\Big(|\partial_s|^{-1/2}\, (\nabla f\cdot n)_{\pm}\Big)^2\,ds\lesssim D.
\end{align}
Hence it remains only to establish~\eqref{decc}. Without loss of generality we consider
$(\nabla f\cdot n)_+$. Let $\zeta:\Gamma\to\R$ be arbitrary. We recall the derivative bound from \eqref{hxone} and the transformation from the proof of lemma \ref{l:2p1}.
Let $\hat{\zeta}:\R\to\R$ be defined via $\hat\zeta(x)=\zeta(x,h(x))$ and let $\hat{\zeta}_{ext}:\R^2_+\to\R$ be its harmonic extension to the right half-plane. Finally, define $\zeta_{ext}:\Omega_+\to\R$ via $\zeta_{ext}(x,z)=\hat\zeta_{ext}(x,z-h(x))$.
Using the dual formulation, one obtains on the one hand
\begin{eqnarray}
\lefteqn{\left(\int_\Gamma \left(|\partial_s|^{-1/2} (\nabla f\cdot n)_+\right)^2\,ds\right)^{1/2}}\notag\\
&&= \sup\left\{\int_\Gamma \zeta\,(\nabla f\cdot n)_+\,ds\colon
\int_{\R^2_+}\abs{\nabla \hat{\zeta}_{ext}}^2\,dx\,d\tilde{z}\leq 1\right\}.\label{dualll}
\end{eqnarray}
On the other hand, one obtains from \eqref{ms} that
\begin{eqnarray}
 \int_\Gamma (\nabla f\cdot n)_+\,\zeta\,ds&\overset{\eqref{ms}}=&\int_{\Omega_+}\nabla f\cdot\nabla\zeta_{ext}
\,dx\,dz\notag\\
 &\leq& \left(D\int_{\Omega_+}|\nabla \zeta_{ext}|^2\,dx\,dz\right)^{1/2}\notag\\
 &\overset{\eqref{hxone}}\lesssim& \left(D\int_{\R^{2}_{+}}|\nabla_{x,\tilde{z}} \hat\zeta_{ext}|^2\,dx\,d\tilde{z}\right)^{1/2}.\label{dual}
\end{eqnarray}
The combination of \eqref{dualll} and \eqref{dual} yields \eqref{decc}.

We will now use \eqref{dec1} -- \eqref{decc} to show that the right-hand side terms in \eqref{deq} satisfy
\begin{eqnarray}
\left| \int_\Gamma (\nabla f\cdot n)_{\pm}\, V^2\,ds\right|&\lesssim& D^{5/6} \left(\int_\Gamma V_s^2\,ds\right)^{2/3}\label{dec7}\\
\text {and } \int_\Gamma \kappa^2 V^2\,ds&\lesssim& E^{1/3}D \left(\int_\Gamma V_s^2\,ds\right)^{2/3},\label{dec3}
\end{eqnarray}
respectively.  This will complete our derivation of \eqref{dd}, since
 inserting \eqref{dec3} and \eqref{dec7} into \eqref{deq} and applying Young's inequality (with exponents $(3,\,\frac{3}{2})$) establishes the estimate. 

On the one hand, the estimate \eqref{dec3} for the quartic term follows directly from
\begin{eqnarray}
 \int_\Gamma \kappa^2 V^2\,ds&\leq&\left(\norm{V}_{\infty}\right)^2\int_{\Gamma}\kappa^2\,ds\notag\\
 &\overset{\eqref{l1.3},\eqref{dec2}}\lesssim& E^{1/3}D \left(\int_\Gamma V_s^2\,ds\right)^{2/3}.\notag
\end{eqnarray}
For \eqref{dec7}, on the other hand,  we again consider without loss of generality $(\nabla f\cdot n)_+$. We estimate by duality
\begin{eqnarray}
\left| \int_\Gamma(\nabla f\cdot n)_{+}\,V^2\,ds\right|&\lesssim&  \norm{|\partial_s|^{1/2}V^2}_{2}
\norm{ |\partial_s|^{-1/2} (\partial_nf)_{+}}_{2}\notag\\
&\overset{\eqref{decc}}\lesssim & \norm{|\partial_s|^{1/2}V^2}_{2} \,D^{1/2}.\label{dec6}
\end{eqnarray}
For the first term, we have by interpolation
\begin{eqnarray*}
 \norm{|\partial_s|^{1/2}V^2}_{2} &\leqslant  &\norm{\partial_s \,V^2}_{2}^{1/2} \norm{V^2}_{2}^{1/2}\lesssim \norm{V}_{\infty}\norm{V_s}_{2}^{1/2} \norm{V}_{2}^{1/2}\\
 &\leqslant& \norm{V}_{\infty}\,\norm{V_s}_{2}^{1/2} \norm{V_s}_{2}^{1/6}\norm{|\partial_s|^{-1/2}V}_{2}^{1/3}\\
 &\overset{\eqref{dec1},\eqref{dec2}}\lesssim& D^{1/3}\norm{V_s}_{2}^{4/3}.
\end{eqnarray*}
Substituting into \eqref{dec6} returns \eqref{dec7}.

Finally, we address \eqref{dhdt}. We calculate, using the definition of $H$ and $V$, that
\begin{eqnarray*}
 \frac{d}{dt}H&=&2\int_{\R^2}\nabla\varphi\cdot\nabla \varphi_t\,\dx=-2\int_{\R^2}\varphi\, \partial_t\Delta\varphi \,\dx\\
 &\overset{\eqref{h}}=&2\int_{\R^2}\varphi\, \partial_t(\textbf{1}_{\Omega_+(t)}-\textbf{1}_{\Omega_0})\,\dx\\
 &=&2\int_{\Gamma} \varphi V\,ds\\
 &\overset{\eqref{vel},\eqref{ms}}=&2\int_{\R^2} f\Delta\varphi\,\dx\\
 &\overset{\eqref{h}}=&-2\int_{\R^2} f(\textbf{1}_{\Omega_+(t)}-\textbf{1}_{\Omega_0})\,\dx\\
 &=&2\int_{\{h>0\}}\int_{0}^{h(x)}f(x,z)\, dzdx-2\int_{\{h<0\}}\int_{h(x)}^{0}f(x,z)\, dzdx.
\end{eqnarray*}
We will estimate the right-hand side. For $\{h> 0\}$, we compute
\begin{align}
\lefteqn{ \int_{\{h>0\}}\,\int_0^{h(x)} f(x,z)\,dz\,dx}\notag\\
 &\overset{\eqref{ms}}=\int_{\{h>0\}}\,\int_0^{h(x)} \big(f(x,z)-f(x,h(x))\big)\,dz\,dx+\int_{\{h>0\}} h\,\kappa\,dx,\label{nobs}
\end{align}
and similarly for $\{h<0\}$. The second term from ${\{h>0\}}$ and $\{h< 0\}$ combine to give
\begin{align*}
 \int_\R h\,\kappa\,dx\overset{\eqref{kaph}}=-\int_\R \frac{h_x^2}{\sqrt{1+h_x^2}}\,dx\overset{\eqref{hxone},\eqref{l1.2}}\sim - E.
\end{align*}
Dropping this negative term, it remains to bound the first term in \eqref{nobs}.
We rewrite the difference as an integral and obtain from Fubini's theorem and H\"older's inequality that
\begin{align*}
\lefteqn{ \left|\int_{\{h>0\}}\,\int_0^{h(x)} \big(f(x,z)-f(x,h(x))\big)\,dz\,dx \right| }\\
 &=\left|\int_{\{h>0\}}\int \int \,\ch_{\{0<z<h(x)\}} \ch_{\{z<z'<h(x)\}} \partial_z f(x,z')\,dz'\,dz\,dx\right|\\
 &=\left|\int_{\{h>0\}} \int \int  \ch_{\{0<z'<h(x)\}} \ch_{\{0<z<z'\}}
  \partial_z f(x,z')\,dz\,dz'\,dx\right|\\
 &\leq\left( \int_{\{h>0\}} \int_0^{h(x)} z'^2\,dz'\,dx\int_{\{h>0\}}\int_0^{h(x)} (\partial_z f(x,z'))^2\,dz'\,dx\right)^{1/2}\\
 &\lesssim \left(\int_\R |h|^3\,dx\int_{\R^2}|\nabla f|^2\,\dx\right)^{1/2}=  \left(\int_\R |h|^3\,dx\,D\right)^{1/2}.
\end{align*}
In light of Lemmas \ref{l:2p1}, \ref{littlewood} and \ref{l:2p4}, this estimate---together with the corresponding one for $\{h<0\}$---yields \eqref{dhdt}.
\end{proof}
\section{Proof of main theorem}\label{S:ode}
In this section, we combine the preceding algebraic and differential relationships with an ODE argument in order to prove Theorem \ref{t:MS}. We begin with an auxiliary lemma that establishes that for a sufficiently small positive $\eps$ and given $|h_x(\cdot,t)|\leq 1$, the set
\begin{align*}
 \{E^2 D\leqslant \eps\}
\end{align*}
is invariant under the evolution \eqref{vel}-\eqref{ms}. In terms of the notation \ref{notn}, assumptions~\eqref{assm} and \eqref{assmh} imply that $(E^{2}D)(t)\ll 1$ for all times $t\geq 0$.
\begin{lemma}\label{l:invariant}
There exists an $\eps\in(0,1)$ such that
\begin{align}
|h_{x}(\cdot, t)|\leqslant 1 \ \text { and } \ (E^2D)(t)\leqslant \eps \;\;\text{imply }\quad \frac{d(E^2D)}{dt}(t)\leq 0.\label{e2d}
\end{align}
\end{lemma}
\begin{proof}
A direct calculation yields
\begin{eqnarray*}
 \frac{d}{dt}(E^2D)&=&2 E\frac{d E}{dt}D+E^2\frac{d D}{dt}\overset{\eqref{ee}}=-2ED^2+E^2\frac{d D}{dt}\\
 &\overset{\eqref{dd}}\leq &-2ED^2+C(E^2D^{5/2}+E^3D^3)\\
 &=&-\Big(2-C((E^2D)^{1/2}+E^2D)\Big)ED^2
\end{eqnarray*}
for some universal constant $C>0$. The smallness of $E^{2}D(t)$ guarantees that the quantity in the brackets is positive.
\end{proof}

We are now ready to prove Theorem \ref{t:MS}.
\begin{proof}[Proof of Theorem \ref{t:MS}]
We start with initial data that satisfies $E_{0}^{2}D_{0}\leqslant \eps$, for $\eps$ given by Lemma~\ref{l:invariant}. To establish estimates \eqref{regime} and \eqref{hxsmall}, it suffices to show that
\begin{align}\label{SmallSlope}
\sup_{\R}|h_{x}(\cdot, t)|<1 \text{ for all } t\geqslant 0.
\end{align}
Indeed, if \eqref{SmallSlope} holds, then \eqref{regime} follows from Lemma~\ref{l:invariant} and the assumption on the initial data, and \eqref{hxsmall} follows from \eqref{regime} via Lemma~\ref{l:stays}.

Define
\begin{align*}
T:=\inf\{t>0:\ \sup_{\R}|h_{x}(\cdot, t)|=1\},
\end{align*}
with the convention that the infimum of an empty set is $+\infty$. By Lemma~\ref{l:stays} and assumptions \eqref{assm} and \eqref{assmh}, we know that $\sup_{\R}|h_{x}(\cdot, 0)|\ll1$, and $T$ is thus positive. We now show that $T=+\infty$, which is equivalent to \eqref{SmallSlope}.

Suppose that $T$ is, on the contrary, finite. Since $h$ is continuous in $t$, we still have that $\sup_{\R}|h_{x}(\cdot, t)|\leqslant 1$ for all $t\in [0, T]$. Therefore Lemma \ref{l:stays} applies for all $t\in[0, T]$ and yields the estimate
\begin{align*}
\sup_{\R}|h_{x}(\cdot, T)|\lesssim (E^{2}D(T))^{1/6}\ll 1,
\end{align*}
a contradiction.
\bigskip

We now turn to \eqref{hless}. Since $E(t)$ is a smooth and strictly decreasing (cf.\ \eqref{ee}) function of~$t$, we can change variables in \eqref{dhdt} and treat $H$ as a function of~$E$. This gives
\begin{align*}
  -\frac{dH}{dE}\lesssim H^{1/2}E^{1/6}D^{-5/12}.
\end{align*}
Inserting the interpolation inequality \eqref{ehd}, we deduce
\begin{align*}
 -\frac{dH}{dE}\lesssim H^{11/12}\,E^{-2/3},
\end{align*}
which we reexpress as
\begin{align*}
 -\frac{d}{dE}(H^{1/12})\lesssim\frac{d}{dE}(E^{1/3}).
\end{align*}
Integrating over an interval $[E, E_{0}]$ (note that this corresponds to an interval $[0, t]$ in the time variable) gives
\begin{align*}
H^{1/12}(E)-H_{0}^{1/12}\leqslant C(E_{0}^{1/3}-E^{1/3}),
\end{align*}
which implies
\begin{align}
 H\lesssim H_0+E_0^4.\label{almost}
\end{align}

Finally, we deduce from \eqref{regime} and \eqref{ehd} that
\begin{align*}
 E^4\lesssim H\,E^2\,D\ll H;\qquad\text{in particular, }\; E_0^4\ll H_0,
\end{align*}
so that \eqref{almost} improves to \eqref{hless}.

\bigskip

The time-independent part of \eqref{eless} follows directly from \eqref{ee}.
We now regard the time-rate of decay from \eqref{eless}. Given \eqref{hless}, the interpolation estimate \eqref{ehd} implies
\begin{align*}
 D\gtrsim H_0^{-1}\,E^2,
\end{align*}
which in light of the differential equation \eqref{ee} yields
\begin{align*}
 \frac{dE}{dt}\lesssim -H_0^{-1}E^2.
\end{align*}
We obtain \eqref{eless} via integration.

\bigskip

Finally, we turn to the decay estimate for the dissipation. On the one hand, the differential equality \eqref{ee},  the preceding result \eqref{eless}, and the positivity of the energy imply that for any $0<t<T$, we have
\begin{align}
 \frac{H_0}{t}\gtrsim E(t)>E(t)-E(T)=\int_t^T D(s)\,ds.\label{one}
\end{align}
On the other hand, the differential inequality \eqref{dd} and the prior result \eqref{regime} yield
\begin{align*}
 \frac{dD}{dt}\lesssim D^{5/2},\qquad\text{i.e., }-\frac{d}{dt}D^{-3/2}(t)\lesssim 1.
\end{align*}
Integrating from $t$ to $T$ and solving for $D(t)$ leads to
\begin{align}
  D(t)\gtrsim \frac{D(T)}{\Big(1+(T-t)D^{3/2}(T)\Big)^{2/3}}.\label{two}
\end{align}
Using \eqref{two} to substitute for $D(s)$ in \eqref{one} yields
\begin{align*}
  \frac{H_0}{t}&\gtrsim\int_t^T\frac{D(T)}{\Big(1+(T-s)D^{3/2}(T)\Big)^{2/3}}\,ds\\
  &=D(T)^{-1/2}\int_0^{D^{3/2}(T)(T-t)}\;\frac{1}{(1+\sigma)^{2/3}}\,d\sigma\\
  &\geqslant \frac{D(T)(T-t)}{\Big(1+(T-t)D^{3/2}(T)\Big)^{2/3}}\\
  &\gtrsim\min\left\{ D(T)(T-t),(T-t)^{1/3}\right\},
  \end{align*}
where the last inequality depends on whether the inequality $D^{3/2}(T)(T-t)\geq 1$ or the opposite holds true.
Setting $t=T/2$ leads to the dichotomy
\begin{align*}
  D(T)\lesssim\frac{H_0}{T^2}\quad\text{or}\quad T^{4/3}\lesssim H_0.
\end{align*}

\end{proof}

\appendix
\section{Homogeneous fractional Sobolev spaces}\label{S:fracsp}
Throughout the appendix we work under the assumption that the curve $\Gamma$ is given by the graph of $h:\R\to\R$ and that
\begin{align}
\sup_\R  |h_x|\leq 1.\label{assh}
\end{align}
We define the operator $|\partial_{x}|^{\sigma}$ with $\sigma \in\R $ acting on the  function $f:\R\to\R$ via
\begin{align}\label{eq:DefFrac}
|\partial_{x}|^{\sigma}f(x)=\int_{\R}|\xi|^{\sigma}\hat{f}(\xi)e^{2\pi i x\xi}\,d\xi,
\end{align}
where $\hat f$ is the Fourier transform of $f$. The fractional Sobolev space $\dot{H}^{\sigma}(\R)$ consists of all locally integrable functions $f$ for which the $L^{2}$-norm of $|\partial_{x}|^{\sigma}f$ is finite. Using the identity
\begin{align}\label{eq:FracNorm}
\norm{|\partial_{x}|^{\sigma}f}_{L^2(\R)}=\left(\int_{\R}|\xi|^{2\sigma}|\hat{f}(\xi)|^{2}\,d\xi\right)^{1/2},
\end{align}
one obtains \textbf{interpolation estimates} among seminorms $\norm{|\partial_{x}|^{\sigma}f}_{2}$ with different $\sigma$ by applying H\"older's inequality with suitable exponents in Fourier space.

The spaces $\dot{H}^{\pm\sigma}(\R)$ are dual to each other:
\begin{align}\label{eq:Duality}
\norm{|\partial_{x}|^{\pm\sigma}f}_{L^2(\R)}=\sup\left\{\int_{\R}f(x)\zeta(x)\, dx \ |\zeta\in \dot{H}^{\mp\sigma}(\R), \norm{|\partial_{x}|^{\mp\sigma}\zeta}_{2}=1\right\}.
\end{align}

Similarly, we will say that $f:\Gamma\to\R$ is in $\dot{H}^\sigma(\Gamma)$ if the function $\hat{f}:\R\to\R$ defined via $$\hat{f}(s)=f(x(s),h(x(s))$$ for arclength parameter $s(x)=\int_{0}^{x}\sqrt{1+h_{x}^{2}(y)}\, dy$ is in $\dot{H}^\sigma(\R)$. In this case we define
\begin{align}\label{eq:frac_sim}
\norm{\abs{\partial_{s}}^{\sigma}f}_{L^2(\Gamma)}:=\norm{\abs{\partial_{s}}^{\sigma}\hat{f}}_{L^2(\R)},
\end{align}
and the \textbf{interpolation estimates} for fractional Sobolev spaces on $\Gamma$ hold true in the same way as for fractional Sobolev spaces on $\R$.
Since $\Gamma$ is the graph of a function $h$ with uniformly bounded slope, the function $\bar{f}:\R\to\R$ defined via $\bar{f}(x):=f(x,h(x))$ satisfies
\begin{align}\label{eq:frac_sim}
\norm{\abs{\partial_{x}}^{\sigma}\bar{f}(x)}_{L^2(\R)}
\sim
\norm{\abs{\partial_{s}}^{\sigma}f(s)}_{L^2(\Gamma)}
\end{align}
and $\bar{f}\in \dot{H}^{\sigma}(\R)$  if and only if $f\in\dot{H}^{\sigma}(\Gamma)$.

It is often useful to think of the $\dot{H}^{1/2}$ norm in terms of harmonic extensions.

\begin{proposition}\label{prop:harmext}
There holds $g\in\dot{H}^{1/2}(\R)$ if and only if there exists $f\in \dot{H}^1(\R^2_+)$ such that $f$ and $g$ satisfy
\begin{align}\label{eq:ext}
\begin {cases} \Delta f=0 &\text{ in }\R^{2}_{+},\\
f=g&\text{ on }\{z=0\}.
\end{cases}
\end{align}
Moreover there holds
\begin{align}
\norm{\abs{\partial_x}^{1/2}g}_{L^2(\R)}= \norm{\nabla f}_{L^2(\R^2_{+})}.\label{exteq}
\end{align}
\end{proposition}
\begin{proof}
According to the Poisson formula, $f$ and $g$ are related via
\begin{align*}
f(x,z)= \int_{-\infty}^{\infty}e^{2\pi i \xi_{1}x-2\pi |\xi_{1}|z}\hat{g}(\xi_{1})\, d\xi_{1}.
\end{align*}
Fourier transforming in the $x$-coordinate and using Fubini's theorem leads to
\begin{align*}
\norm{\nabla f}_{L^2(\R^2_{+})}^2&= \int_{-\infty}^{\infty}|\xi_{1}| \,|\hat{g}(\xi_{1})|^{2}\, d\xi_{1} = \norm{\abs{\partial_x}^{1/2}g}_{L^2(\R)}^{2}.
\end{align*}
\end{proof}

\subsection*{Acknowledgments}
Olga Chugreeva was
partially supported by DFG Grant WE 5760/1-1.

\end{document}